\documentclass[12pt,a4paper,reqno]{amsart} 
\usepackage{amssymb}
\usepackage{latexsym}
\usepackage{amsmath}
\usepackage{mathrsfs}
\usepackage[X2,T1]{fontenc}
\usepackage{cite}

\usepackage{calc}                   
\date{\today}

\newcommand{\scal}[2]{\langle #1,#2\rangle}
\newcommand{\rr}[1]{\mathbf R^{#1}}

\newcommand{\nm}[2]{\Vert #1\Vert _{#2}}

\newcommand{\ep}{\varepsilon}
\newcommand{\fy}{\varphi}
\newcommand{\cdo}{\, \cdot \, }
\newcommand{\supp}{\operatorname{supp}}
\newcommand{\eabs}[1]{\langle #1\rangle}     
\newcommand{\vrum}{\vspace{0.1cm}}
\newcommand{\back}[1]{\setminus #1}
\newcommand{\loc}{\operatorname{loc}}

\newcommand{\WF}{\operatorname{WF}}
\newcommand{\re}{\operatorname{Re}}
\newcommand{\singsupp}{\operatorname{sing \, supp}}
\newcommand{\im}{\operatorname{Im}}
\newcommand{\FB}{\mathscr F\!\mathscr B}

\newcommand\leftidx[3]{%
{\vphantom{#2}}#1#2#3%
}

\newcommand{\ltrans}[1]{\leftidx{^\mathrm{t}}{\!#1}{}}

\numberwithin{equation}{section}          

\newtheorem{thm}{Theorem}
\numberwithin{thm}{section}
\newtheorem*{tom}{\rubrik}
\newcommand{\rubrik}{}
\newtheorem{prop}[thm]{Proposition}
\newtheorem{cor}[thm]{Corollary}
\newtheorem{lemma}[thm]{Lemma}

\theoremstyle{definition}

\newtheorem{defn}[thm]{Definition}

\theoremstyle{remark}

\newtheorem{rem}[thm]{Remark}              %

\author{Karoline Johansson}

\address{Department of Computer science, Physics and Mathematics, Linn{\ae}us University, V{\"a}xj{\"o}, Sweden\footnote{Former address: Department of Mathematics and Systems Engineering,
V{\"a}xj{\"o} University, Sweden}}

\email{karoline.johansson@lnu.se}

\title{Association between temperate distributions and analytical functions in the context of wave-front sets}

\keywords{Wave-front, Fourier, Banach function space}

\subjclass[2000]{35A18}

\begin{document}

\begin{abstract}
Let $\mathscr B$ be a translation invariant Banach function space (BF-space). In this paper we prove that every temperate distribution $f$ can be associated with a function $F$ analytic in the convex tube $\Omega=\{z\in \mathbf C^d;\, |\im z|<1\}$ such that the wave-front set of $f$ of Fourier BF-space types in intersection with $\rr d \times S^{d-1}$ consists of the points $(x,\xi)$ such that $F$ does not belong to the Fourier BF-space at $x-i\xi$.
\end{abstract}

\maketitle

\section{Introduction}\label{sec0}
Wave-front sets of Fourier Banach function types where introduced
by Coriasco, Johansson and Toft in \cite{CJT1}. Roughly speaking,
the wave-front set of Fourier Banach function type, $\WF_{\FB}(f)$, of a
distribution $f$, consists of all pairs
$(x_0,\xi_0)$ such that no localization of the distribution $f$ at
$x_0$ belongs to $\FB$ in the direction $\xi_0$. Several
properties of classical wave-front sets (with respect to
smoothness) can be found in Hörmander \cite{Ho1}. One of these are
mapping properties for pseudo-differential operators (with smooth
symbols) on wave-front sets which were generalized to
Fourier Lebesgue type by Pilipovic, Teofanov
and Toft in \cite{PTT1}. These properties were also proved to hold
for wave-front sets of Fourier Banach function types. (Cf. Coriasco, Johansson and Toft
\cite{CJT1}.)

In this paper we consider another property of wave-front sets
concerning association between a temperate distribution and an
analytic function, which was proved for classical wave-front sets
by Hörmander in \cite{Ho1}. More precisely, Hörmander showed that
every temperate distribution $f$ can be associated with a
function $F$ analytic in the convex tube $\{z\in \mathbf C^d; \,
|\im z|<1\}$ such that
\begin{equation}\label{fochF}
f=\int_{|\xi|=1} \, F(\cdot +i\xi)\, d\xi,
\end{equation}
and
\begin{multline*}
(\rr d\times S^{d-1})\cap \WF_L(f)\\[1 ex]=\{(x,\xi);\, |\xi|=1,\, F \, \text{is not in} \, C^L \, \text{at} \, x-i\xi\}.
\end{multline*}
Here $\WF_L(f)$ is the wave-front set with respect to a class of smooth functions $C^L$. (Cf. Section 8.4 in Hörmander \cite{Ho1}.)

\par

In this paper we generalize this result to wave-front sets of Fourier Banach function types. We show that for every temperate distribution $f$ there exists a function $F$ with the properties given before, satisfying \eqref{fochF} and such that
\begin{multline}\label{vagfront M}
(\rr d\times S^{d-1})\cap \WF_{\FB}(f)\\[1 ex]=\{(x,\xi);\, |\xi|=1,\, F \, \text{is not in}  \, \FB \, \text{at} \, x-i\xi\}.
\end{multline}

\par

Since every Lebesgue space is a Banach function space we get by
choosing $\mathscr B=L^p$ that the analogous result for wave-front sets of Fourier Lebesgue types is contained in \eqref{vagfront M} as a special case.

\par

As shown later on in this paper, analogous results hold also for the weighted cases as well as inf types and modulation space types of wave-front sets. The latter is a direct consequence of the identification of wave-front sets of Fourier BF-spaces types with wave-front sets of modulation space types.

\par

The modulation spaces were introduced by Feichtinger in \cite{F1}, and
the theory was developed in
\cite{Feichtinger3, Feichtinger4, Feichtinger5, Grochenig0a}. The
modulation space $M(\omega, \mathscr B )$, where $\omega$ is
an appropriate weight function (or time-frequency shift) on phase space $\rr {2d}$,
appears as the set of temperate (ultra-)distributions
whose short-time Fourier transform belong to the weighted Banach space
$\mathscr B(\omega )$.
This family of modulation spaces contains the (classical) modulation
spaces $M^{p,q}_{(\omega)}(\rr {2d})$ as well as the space
$W^{p,q}_{(\omega)}(\rr {2d})$ related to the Wiener amalgam spaces. In fact, these spaces which occur frequently in the time-frequency community are obtained
by choosing $\mathscr B =L^{p,q}_1(\rr {2d})$ or $\mathscr B
=L^{p,q}_2(\rr {2d})$ (see Remark 6.1 in \cite{CJT1}).

\par

The paper is organized as follows. In Section \ref{sec1} we recall
the definitions and some basic properties for translation
invariant Banach function spaces (BF-spaces) and Fourier Banach
function spaces. In Section \ref{sec2} we prove that every
temperate distribution $f$ can be associated with a function $F$
analytic in a convex tube satisfying \eqref{fochF} and
\eqref{vagfront M}. Analogous results are given in Sections \ref{sec2'}-\ref{sec2'''} for the weighted case, inf types and modulation space types, respectively. We use this result in Section \ref{sec3} to
show some further properties of these wave-front sets. In
particular we show a result about the relation between wave-front
sets of Fourier Banach function types and analytic wave-front sets.

\par

\section{Preliminaries}\label{sec1}
\par

In this section we recall some notations and basic results. The
proofs are in general omitted. In what follows we let $\Gamma$
denote an open cone in $\rr d\back 0$. If $\xi \in \rr d\back 0$
is fixed, then an open cone which contains $\xi $ is sometimes
denoted by $\Gamma_\xi$.

\par

Assume that $\omega, v\in L^\infty _{loc}(\rr d)$ are positive
functions. Then $\omega$ is called $v$-moderate if
\begin{equation}\label{moderate}
\omega (x+y) \leq C\omega (x)v(y)
\end{equation}
for some constant $C$ which is independent of $x,y\in \rr d$. If $v$
in \eqref{moderate} can be chosen as a polynomial, then $\omega$ is
called polynomially moderate. We let $\mathscr P(\rr d)$ be the set
of all polynomially moderated functions on $\rr d$. We say that $v$ is
\emph{submultiplicative} when \eqref{moderate} holds with $\omega =v$.
Throughout we assume that the submultiplicative weights are even.
If $\omega (x,\xi )\in \mathscr P(\rr {2d})$ is constant with respect
to the $x$-variable ($\xi$-variable), then we sometimes write $\omega
(\xi )$ ($\omega (x)$) instead of $\omega (x,\xi )$. In this case we
consider $\omega$ as an element in $\mathscr P(\rr {2d})$ or in
$\mathscr P(\rr d)$ depending on the situation.

\par
For any weight $\omega$ in $\mathscr P(\rr d)$ we let $L^p_{(\omega )}(\rr d)$ be the set of all
$f\in L^1_{loc}(\rr d)$ such that $f\cdot \omega \in L^p(\rr d)$.

\medspace

The Fourier transform $\mathscr F$ is the linear and continuous
mapping on $\mathscr S'(\rr d)$ which takes the form
$$
(\mathscr Ff)(\xi )= \widehat f(\xi ) \equiv (2\pi )^{-d/2}\int _{\rr
{d}} f(x)e^{-i\scal  x\xi }\, dx
$$
when $f\in L^1(\rr d)$. We recall that $\mathscr F$ is a homeomorphism
on $\mathscr S'(\rr d)$ which restricts to a homeomorphism on $\mathscr
S(\rr d)$ and to a unitary operator on $L^2(\rr d)$.

\medspace

Next we recall the definition of Banach function spaces.

\par

\begin{defn}\label{BFspaces}
Assume that $\mathscr B$ is a Banach space of complex-valued
measurable functions on $\rr d$ and that $v \in \mathscr P(\rr {d})$
is submultiplicative.
Then $\mathscr B$ is called a \emph{(translation) invariant
BF-space on $\rr d$} (with respect to $v$), if there is a constant $C$
such that the following conditions are fulfilled:
\begin{enumerate}
\item $\mathscr S(\rr d)\subseteq \mathscr
B\subseteq \mathscr S'(\rr d)$ (continuous embeddings);

\vrum

\item if $x\in \rr d$ and $f\in \mathscr B$, then $f(\cdot -x)\in
\mathscr B$, and
\begin{equation*}
\nm {f(\cdot -x)}{\mathscr B}\le Cv(x)\nm {f}{\mathscr B}\text ;
\end{equation*}

\vrum

\item if $f,g\in L^1_{loc}(\rr d)$ satisfy $g\in \mathscr B$  and $|f|
\le |g|$ almost everywhere, then $f\in \mathscr B$ and
$$
\nm f{\mathscr B}\le C\nm g{\mathscr B}\text .
$$
\end{enumerate}
\end{defn}

\par

Assume that $\mathscr B$ is a translation invariant BF-space. If $f\in
\mathscr B$ and $h\in L^\infty$, then it follows from (3) in
Definition \ref{BFspaces} that $f\cdot h\in \mathscr B$ and
\begin{equation*}
\nm {f\cdot h}{\mathscr B}\le C\nm f{\mathscr B}\nm h{L^\infty}.
\end{equation*}

\par

\begin{rem}\label{newbfspaces}
Assume that $\omega _0,v,v_0\in \mathscr P(\rr d)$ are such $v$ and
$v_0$ are submultiplicative,
$\omega _0$ is $v_0$-moderate, and assume that $\mathscr B$ is a
translation-invariant BF-space on $\rr d$ with respect to $v$. Also
let $\mathscr B_0$ be the Banach space which consists of all $f\in
L^1_{loc}(\rr d)$ such that $\nm f{\mathscr B_0}\equiv \nm {f\, \omega _0
}{\mathscr B}$ is finite. Then $\mathscr B_0$ is a translation
invariant BF-space with respect to $v_0v$.
\end{rem}

\par

For future references we note that if $\mathscr B$ is a translation
invariant BF-space with respect to the submultiplicative weight $v$ on
$\rr d$, then the convolution map $*$ on $\mathscr S(\rr d)$ extends
to a continuous mapping from $\mathscr B\times L^1_{(v)}(\rr
d)$ to $\mathscr B$, and for some constant $C$ it holds
\begin{equation}\label{propupps}
\nm {\fy *f}{\mathscr B}\le C\nm \fy{L^1_{(v)}}\nm f{\mathscr B},
\end{equation}
when $\fy \in L^1_{(v)}(\rr d)$ and $f\in \mathscr B$. In fact, if $f,g\in \mathscr S$, then $f*g\in \mathscr S\subseteq \mathscr B$ in view of the definitions, and
Minkowski's inequality gives
\begin{multline*}
\nm {f*g}{\mathscr B} =\Big \Vert \int f(\cdo -y)g(y)\, dy \Big \Vert
_{\mathscr B}
\\[1ex]
\le \int \nm { f(\cdo-y)}{\mathscr B}|g(y)|\, dy \le C\int \nm {
f}{\mathscr B}|g(y)v(y)|\, dy = C\nm { f}{\mathscr B}\nm
g{L^1_{(v)}}.
\end{multline*}

\par

Since $\mathscr S$ is dense in $L^1_{(v)}$, it follows that $\fy *f\in \mathscr B$ when $\fy \in L^1_{(v)}$ and $f\in \mathscr S$, and that \eqref{propupps} holds in this case. The result is now a consequence of Hahn-Banach's theorem.

\par

From now on we assume that each translation invariant BF-space $\mathscr B$ is such that the convolution map $*$ on $\mathscr S(\rr d)$ is \emph{uniquely} extendable
to a continuous mapping from $\mathscr B\times L^1_{(v)}(\rr
d)$ to $\mathscr B$, and that \eqref{propupps} holds when $\fy \in L^1_{(v)}(\rr d)$ and $f\in \mathscr B$. We note that $\mathscr B$ can be any mixed and weighted Lebesgue space. 

\par

In particular we then have that
$$
\nm{\int_{|y|=1} g(\cdot-y)\, dy}{\FB} \leq C \int_{|y|=1} \nm{g(\cdot-y)}{\FB} \, dy.
$$

\par

Assume that $\mathscr B$ is a translation invariant BF-space on $\rr
d$ and $\omega\in \mathscr P(\rr d)$. Then we let $\FB(\omega)$ be the set of all $f\in \mathscr S'(\rr d)$ such that
$\xi \mapsto \widehat f(\xi )\omega(\xi)$ belongs to $\mathscr
B$. It follows that $\FB(\omega) $ is a Banach space under the
norm
\begin{equation*}
\nm f{\FB(\omega)}\equiv \nm {\widehat
f \omega}{\mathscr B}.
\end{equation*}

\par

Recall that a topological vector space $V\subseteq \mathscr
D'(X)$ is called \emph{local} if $V\subseteq V_{loc}$. Here
$X\subseteq \rr d$ is open, and $V_{loc}$ consists of all $f\in
\mathscr D'(X)$ such that $\fy f \in V$ for every $\fy \in C_0^\infty
(X)$. For future references we note that if $\mathscr B$ is a
translation invariant BF-space on $\rr d$, then it follows from \eqref{propupps} that $\FB$ is a local space, i.{\,}e.
\begin{equation*}
\FB\subseteq \FB_{loc}\equiv (\FB)_{loc}.
\end{equation*}

\par
Let
\begin{equation}\label{IochK}
I(\xi)=\int_{|\omega|=1} e^{-\scal \omega \xi}\, d \omega \qquad
\text{and} \qquad K(z)=(2\pi)^{-d}\int e^{i\scal z\xi}/I(\xi)\,
d\xi. \end{equation}
These functions will play an important role when proving the main results. We therefore explicitly give properties of these functions. These results can be found in Section 8.4 in Hörmander \cite{Ho1}.

Let $I$ be given by \eqref{IochK} then we have that $I(\xi)=2\cosh \xi$ for $d=1$ and $I(\xi)=I_0(\scal \xi \xi^{1/2})$ for $d>1$. Here 
\begin{equation}
I_0(\rho)=c_{d-1}\int_{-1}^1 (1-t^2)^{(d-3)/2}e^{t\rho}\, dt,
\end{equation}
where $c_{d-1}$ is the area of $S^{d-2}$. Then $I_0$ is an even analytic function in $\mathbf C$ such that for every $\ep>0$ 
$$
I_0(\rho) =(2\pi)^{(d-1)/2}e^{\rho}\rho^{-(d-1)/2}(1+O(1/\rho))
$$ 
if $\rho\to\infty$, $|\arg \rho|<\pi/2-\ep.$
Furthermore there is a constant $C$ such that for all $\rho\in \mathbf C$ we have that 
$$
|I_0(\rho)|\leq C(1+|\rho|)^{-(d-1)/2}e^{|\operatorname{Re} \rho|}.
$$
The following lemma can be found with proof in \cite{Ho1}. 
\begin{lemma}
$K(z)$ is an analytic function in the connected open set 
$$
\widetilde{\Omega}=\{z\in \mathbf C^d;\, \scal z z \not\in (-\infty,-1]\}\supset \Omega.
$$
Here $\Omega=\{z\in \mathbf C;\, |\operatorname{Im} z|<1\}$.
Furthermore, for any closed open cone $\Gamma\subset \widetilde{\Omega}$ such that $\scal z z $ is never $\leq 0$ when $z\in \Gamma\back 0$ there is some $c>0$ such that $K(z)=O(e^{-c|z|})$ when $z\to \infty$ in $\Gamma$. We have for real $x$ and $y$ that 
\begin{equation}
|K(x+iy)|\leq K(iy)=(d-1)!(2\pi)^{-d}(1-|y|)^{-d} (1+O(1-|y|)),
\end{equation}
$|y|\to 1^-$. 
\end{lemma}

\par

Furthermore 
$$
|D^\beta K(x+iy)|\leq C_\beta (1-|y|)^{-n-|\beta|}e^{-c|x|},\qquad |y|<1
$$
holds by Cauchy's inequalities. (Cf. \cite{Ho1}.)
\par

\section{Analytic functions associated with temperate distributions}\label{sec2}

In this section we show that \eqref{vagfront M} holds. Assume that
$\omega\in \mathscr P(\rr d)$. We recall that the
wave-front sets of weighted Fourier Banach function types
$\WF_{\FB(\omega)}(f)$ consists of all pairs $(x_0,\xi_0)\in \rr d
\times \rr d\back 0$ such that $$ |\fy f|_{\FB(\omega,
\Gamma_{\xi_0})}\equiv \nm{\mathscr F (\fy f)
\chi_{\Gamma_{\xi_0}}\omega}{\mathscr B} =\infty, $$ for every
open conical neighbourhood $\Gamma_{\xi_0}$ of $\xi_0$, and
$\fy\in C^{\infty}_0$ with $\fy=1 $ in some open neighbourhood $X$
of $x_0$. Here $\chi_{\Gamma_{\xi_0}}$ is the characteristic
function of $\Gamma_{\xi_0}$. (Cf. Coriasco, Johansson and Toft
\cite{CJT1}.)

\par

Let $\WF_{\FB(\omega)}(f)=\WF_{\FB}(f)$ if $\omega \equiv 1$.

\par
\begin{defn}\label{FB i x}
Assume that $f\in \mathscr D'(\rr d)$, $\mathscr B$ is a
translation invariant BF-space and $\omega \in \mathscr P(\rr d)$.
Then $f\in \mathcal \FB(\omega)$ at $x_0$ if and only if there exists $\fy\in
C^{\infty}_0$ with $\fy\equiv 1$ in a neighbourhood of $x$ such that $\fy f \in
\mathscr \FB(\omega)$.
\end{defn}
\begin{rem}
For convenience we say that $f\in \mathcal \FB$ at $x_0$ if the
statement in Definition \ref{FB i x} is true for $\omega\equiv 1$.
\end{rem}
We note that if $f$ belongs to $\FB(\omega)$ at $x_0$ then $(x_0,\xi_0)\notin\WF_{\FB(\omega)}(f)$ for any $\xi_0\in \rr d\back 0.$

\begin{defn}
For $f\in \mathscr D'(X)$ the singular support $\singsupp_{\mathscr B(\omega)} f$ is the smallest closed subset of $X$ such that $f$ is in $\FB(\omega)$ in the complement.
\end{defn}

We use the notation $\singsupp_{\mathscr B(\omega)}
f=\singsupp_{\mathscr B} f$ when $\omega\equiv 1$.
\begin{thm}
Assume that $f\in \mathscr D'(\rr d)$, $\mathscr B$ is a
translation invariant BF-space and $\omega\in\mathscr P (\rr
d)$. The projection of $\WF_{\FB(\omega)} (f)$ in $X$ is equal to
$\singsupp_{\mathscr B(\omega)} f$.
\end{thm}

\begin{proof}
(a) Assume that $x_0\not\in \singsupp_{\mathscr B (\omega)}(f)$. Then $f$ belongs to $\FB(\omega)$ at $x_0$. This implies that $(x_0,\xi_0)\not\in\WF_{\FB(\omega)} (f)$, for any $\xi_0\in \rr d\back 0$.

\medspace

(b) Assume that $(x_0,\xi_0)\not\in\WF_{\FB(\omega)} (f)$ for all
$\xi_0\in \rr d\back 0$. Then by the compactness of unit sphere we can choose a neighbourhood $K$ of
$x_0$ such that $\WF_{\FB(\omega)} (f)\cap (K\times \rr
d)=\emptyset$. This implies that we can choose a function
$\fy_{x_0}\in C^\infty _0$ which is equal to $1$ in a
neighbourhood $X$ of $x _0$ such that $\fy_{x_0} f \in \mathscr
\FB(\omega)$. Hence $x_0\not\in \singsupp_{\mathscr B(\omega)}
(f)$.
\end{proof}

The next theorem is given without proof since the result follows
directly from Theorem 8.4.8 in Hörmander \cite{Ho1} together with
the observation that $\WF_{\FB}(f)\subseteq \WF(f)$.
\begin{thm}\label{8.4.8}
Let $X\subseteq \rr d$ be open, $\Gamma$ an open convex cone in $\rr d$ and let
$$
Z=\{z\in  \mathbf C^d;\, \re\, z\in X,\, \im\, z\in \Gamma,\, |\im \, z|<\gamma\},
$$
for some $\gamma>0$. Also let $F$ be an analytic function in $Z$ such that
$$
|F(z)|\leq C|\im \, z|^{-N},\, z\in Z.
$$
Then $F(\cdot +iy)$ has the limit $F_0\in \mathscr D'^{N+1}(X)$ as $y\in \Gamma$ tends to zero and $\WF_{\FB}(F_0)\subset X\times (\Gamma^\circ \back 0)$, where $\Gamma^\circ$ is the dual cone of $\Gamma$. Furthermore $F=0$ if $F_0=0$. 
\end{thm}

\par

Next we associate the temperate distribution $f$ with a function $F$ analytic in the convex cone $\Omega=\{z\in \mathbf C^d; \, |\im z|<1\}$ such that
$$
f=\int_{|\xi|=1} F(\cdot +i\xi)\, d\xi.
$$

We recall the following result from Hörmander \cite[Theorem 8.4.11]{Ho1}.

\begin{thm}\label{8.4.11}
Let $K$ be given by \eqref{IochK}. If $f\in \mathscr S'(\rr d)$
and $F=K*f$, then $F$ is analytic in $\Omega=\{z; \, |\im\, z|
<1\}$ and for some $C,a,b$
\begin{equation}\label{rel1fF}
|F(z)|\leq C(1+|z|)^a(1-|\im\, z|)^{-b}, \, z\in \Omega.
\end{equation}
The boundary values $F(\cdot + i \xi)$ are continuous functions of $\xi\in S^{d-1}$ with values in $\mathscr S'(\rr d)$, and
\begin{equation}\label{rel2fF}
\scal{f}{\phi}=\int\scal{F(\cdot+i\xi)}{\phi}\, d\xi,\qquad \phi\in \mathscr S.
\end{equation}
Conversely, if $F$ satisfies \eqref{rel1fF}, then \eqref{rel2fF} defines a distribution $f\in \mathscr S'$ with $F=K*f$.
\end{thm}

\par

Next we give the main theorem.

\begin{thm}\label{huvudsats}
Assume that $f$ and $F$ satisfies the conditions in Theorem \ref{8.4.11}. Then we have that
$$
(\rr d\times S^{d-1})\cap \WF_{\FB}(f)=\{(x,\xi); \, |\xi|=1,\, F \text{ is not in} \, \FB \text{ at } x-i\xi\}.
$$
\end{thm}

We remark that $F$ is in $\FB$ at $x-i\xi$ if for some neighbourhood $V$ of $(x,\xi)$ there exists some localization $\fy \in C^{\infty}_0$ with $\fy=1$ in $V$ such that $\fy f\in\FB$. Before the proof we note that $C^L$ is a subset of $\FB$.

\par

\begin{proof}
First assume that $(x_0,\xi_0)\not\in \WF_{\FB}(f)$ and $|\xi_0|=1$.
Then we want to show that $F=K*f\in \FB$ at $x_0-i\xi_0$. By the
hypothesis there exist $r>0$ and $\fy_{x_0}\in C^{\infty}_0$ such
that $\fy_{x_0}(x)= 1$ if $|x-x_0|<r$ and an open conical
neighbourhood $\Gamma_{\xi_0}$ of $\xi_0$ such that $$
\nm{\mathscr F(\fy_{x_0} f)\chi_{\Gamma_{\xi_0}}}{\mathscr B}
<\infty. $$ We also recall that since $\fy_{x_0}f$ has compact supports it
holds $$ |\mathscr F(\fy_{x_0} f)(\xi)|\leq C(1+
|\xi|)^M,\qquad \xi\in \rr d, $$ for some fixed constants $C,M\geq 0$. Set $f=\fy_{x_0} f + v$ where $v=f(1-\fy_{x_0})$. Then
$F=K*f=K*(\fy_{x_0} f) + K*v$ and
\begin{equation*}
K*v(z)=\scal {K(z-\cdot)}{v}.
\end{equation*}
Now $K(x+iy-t)$ is well-defined when $|y|^2 <1 + |x-t|^2$, so it is well-defined and rapidly decreasing with all derivatives when $|t-x_0|\geq r$ if
\begin{equation}\label{set1}
|y|^2 <1+(r-|x-x_0|)^2, \qquad |x-x_0|<r.
\end{equation}
(Cf. Lemma 8.4.10 and Theorem 8.4.11 in \cite{Ho1}.)
It follows that $K*v$ is analytic and bounded in compact subsets of the set defined by \eqref{set1}, which is a neighbourhood of $x_0-i\xi_0$. Then it follows that $K*v$ belongs to $\FB$ at $x_0-i\xi_0$.

\par

Next we consider $K*(\fy_{x_0} f)$. It is left to prove that $K*(\fy_{x_0}f)$ belongs to $\FB$ at $x_0- i\xi_0$.
The Fourier transform of $K*(\fy_{x_0} f)(\cdot+iy)$ is $e^{-\scal y \xi} \mathscr F(\fy_{x_0} f)/I(\xi)$. By (8.4.12) in Hörmander \cite{Ho1} it follows that
$$
\frac{1}{|I(\xi)|} \leq C e^{-|\xi|} (1+|\xi|)^{(d-1)/2}.
$$
Using this we conclude that
\begin{multline}\label{uppdeln}
\nm{K*(\fy_{x_0} f)}{\FB} =\nm{e^{-\scal y \cdot} \mathscr F(\fy_{x_0} f)/I(\cdot)}{\mathscr B} \\[1 ex]\leq C_1 \nm{e^{-\scal y \cdot -|\cdot|} (1+|\cdot|)^{(d-1)/2}\mathscr F(\fy_{x_0} f)}{\mathscr B} \\[1 ex]
\leq C_2(\nm{e^{-\scal y \cdot -|\cdot|} (1+|\cdot|)^{(d-1)/2}\mathscr F(\fy_{x_0} f)\chi_{\Gamma_{\xi_0}}}{\mathscr B}+\\[1 ex]\nm{e^{-\scal y \cdot -|\cdot|} (1+|\cdot|)^{(d-1)/2}\mathscr F(\fy_{x_0} f)(1-\chi_{\Gamma_{\xi_0}})}{\mathscr B})
\end{multline}
For the first part in the right-hand side of \eqref{uppdeln} we recognize that for every $y$ such that $|y|<1$
$ \sup_\xi e^{-\scal y \xi -|\xi|} (1+|\xi|)^{(d-1)/2} <\infty$ and therefore
\begin{equation*}
\nm{e^{-\scal y \cdot -|\cdot|} (1+|\cdot|)^{(d-1)/2}\mathscr F(\fy_{x_0} f)\chi_{\Gamma_{\xi_0}}}{\mathscr B}<\infty.
\end{equation*}
Then for the second part we have that
\begin{multline*}
\nm{e^{-\scal y \cdot -|\cdot|} (1+|\cdot|)^{(d-1)/2}\mathscr F(\fy_{x_0} f)(1-\chi_{\Gamma_{\xi_0}})}{\mathscr B}\\[1 ex]\leq
C\nm{e^{-\scal y \cdot -|\cdot|} (1+|\cdot|)^{M+(d-1)/2}(1-\chi_{\Gamma_{\xi_0}})}{\mathscr B}.
\end{multline*}
Choose $\varepsilon >0$ such that $\scal {\xi_0} \xi <(1-2\varepsilon)|\xi|$ when $\xi\notin\Gamma_{\xi_0}$. Then
$$
\scal y \xi +|\xi|>\varepsilon |\xi|
$$
if $\xi \notin \Gamma_{\xi_0}$ and $|y+\xi_0|<\varepsilon$.
Hence we obtain
\begin{multline*}
\nm{e^{-\scal y \cdot -|\cdot|} (1+|\cdot|)^{M+(d-1)/2}(1-\chi_{\Gamma_{\xi_0}})}{\mathscr B} \\[1 ex]\leq C \nm{e^{-\varepsilon |\cdot|} (1+|\cdot|)^{M+(d-1)/2}}{\mathscr B} <\infty
\end{multline*}
This completes the first part of the proof.
\end{proof}

For the second part of the proof we need the following lemma.

\begin{lemma}\label{andra riktningen}
Let $d \mu$ be a measure on $S^{d-1}$ and $\Gamma$ an open convex cone such that
\begin{equation*}
\scal y \xi <0\, \text{ when } \, 0\neq y\in \overline{\Gamma},\, \xi\in \supp d\mu.
\end{equation*}
If $F$ is  analytic in $\Omega$ and satisfies \eqref{rel1fF}, then
$$ F_1(z)=\int F(z+i\xi) \, d\mu(\xi) $$ is analytic and
$|F_1(z)|\leq C'(1+|\re z|)^a|\im z|^{-b}$ when $\im z\in \Gamma$
and $|\im z|$ is small enough.

\par

For every measure $d\mu$ on $S^{d-1}$ we have
\begin{equation}\label{omvand}
\WF_{\FB}(F_\mu)\subset \{(x,\zeta);\, -\zeta/|\zeta| \in \supp d\mu \text{ and } F\not\in \FB \, \text{at}\, x-i\zeta/|\zeta|\}
\end{equation}
Here $F_\mu=\int F(\cdot + i\xi)\, d\mu(\xi)$.
\end{lemma}
\begin{proof} The first statement was proved by Hörmander in \cite[Theorem 8.4.12]{Ho1}. Let $\Gamma^\circ$ be the dual cone of $\Gamma$. By Theorem \ref{8.4.8} it follows that
$$
\WF_{\FB}(F_\mu)\subset \rr d\times \Gamma^\circ.
$$

\par

Assume that $$ x_0\in \{x;\, F\in \FB \,\text{at}\, x+i\xi\,
\text{for every}\, \xi\in \supp d\mu\}. $$ Then we have that for
every $\xi_0\in \supp d\mu$ there exists an open neighbourhood
$U_{\xi_0}$ of $(x_0,\xi_0)$ and a function $\fy_{\xi_0}\in
C^{\infty}_0$ with $\supp \fy_{\xi_0} \subseteq U_{\xi_0}$ such
that $x\mapsto \fy_{\xi_0} F \in \FB$. Since the set $$
\{x_0+i\xi;\,\xi\in\supp d\mu \} $$ is compact, it follows from
arguments about compactness that there exist finitely many points
$\xi_j$ such that $$ \{x_0+i\xi;\,\xi\in\supp d\mu \}\subseteq
\bigcup U_{\xi_j}. $$ For every $\xi_j$ we choose an open
neighbourhood $X_{\xi_j}$ of $x_0$ and let $X=\bigcap X_{\xi_j}$.
Then we can choose $\fy_0\in C^{\infty}_0$  equal to one in the neighbourhood
$X$ of $x_0$ such that $x\mapsto \fy_0  F\in \FB$.
Furthermore, we have that there exists $\fy \in C^\infty_0$, with support in a neighbourhood of $x_0$, such
that
\begin{multline*}
\nm{\fy F_\mu}{\FB}=\nm{\int \fy F(\cdot+i\xi)\, d \mu}{\FB} \\[1 ex]\leq \int\nm{\fy F(\cdot+i\xi)}{\FB}\, d\mu <\infty.
\end{multline*}
Then $\fy F_\mu \in \FB$ at $x_0$.

\par

From the arguments above it follows that
\begin{equation*}
\singsupp_{\mathscr B} (F_\mu) \subset \{x; \, F \, \text{is not in} \, \FB \, \text{at} \, x+i\xi \, \text{for some}\, \xi\in \supp \, d\mu\}.
\end{equation*}
Then we may write $d\mu=\sum d\mu_j$ where $\supp d\mu_j$ is contained in the intersection of $\supp d\mu$ and a narrow open convex cone $V_j$. Applying the result just proved with $d \mu$ replaced by $d \mu_j$ and $\Gamma$ replaced by the interior of the dual cone $-V_j^\circ$ we obtain
$$
\WF_{\FB}(F_\mu)\subset \bigcup \{(x,\zeta); \, -\zeta/|\zeta|\in \overline{V}_j,\, F\notin \FB \, \text{at} \, x+i\xi \, \text{for some} \, \xi\in V_j\}.
$$
If $-\zeta/|\zeta|\notin \supp d \mu$ or $F\in \FB$ at $x-i\zeta/|\zeta|$ we can choose the covering so that $-\zeta/|\zeta|\notin \overline{V}_j$ for every $j$ or for all $j\neq 1$ while $F\in \FB$ at $x+i\xi$ for every $\xi\in V_1$. In both cases it follows that $(x,\xi)\notin \WF_{\FB}(F_\mu)$  which proves \eqref{omvand}.
This completes the proofs of Lemma \ref{andra riktningen} and Theorem \ref{huvudsats}.
\end{proof}

\par

The following Corollary is an analogue to Corollary 8.4.13 in Hörmander \cite{Ho1}.

\begin{cor}\label{8.4.13'}
Let $\Gamma_1,\dots ,\Gamma_m$ be closed cones in $\rr d\back 0$
such that $$ \bigcup_{j=1}^m \Gamma_j = \rr d\back 0. $$ For every
$f\in \mathscr S'(\rr d)$ there exists a  decomposition
$f=\sum_{j=1}^m f_j$, where $f_j\in \mathscr S'$ and
\begin{equation}\label{1}
\WF_{\FB} (f_j) \subseteq \WF_{\FB} (f) \cap (\rr d \times \Gamma_j).
\end{equation}
If there exists another decomposition $f=\sum_{j=1}^m f_j'$ which
also satisfies the conditions above, then $f_j'=f_j + \sum_{k=1}^m
f_{jk}$ where $f_{jk}\in \mathscr S'$, $f_{jk}=-f_{kj}$ and
\begin{equation}\label{2}
\WF_{\FB} (f_{jk}) \subset WF_{\FB}(f) \cap (\rr d \times (\Gamma_j \cap \Gamma_k)).
\end{equation}
\end{cor}

\begin{proof}
Let $\phi_j$ be the characteristic function on $\Gamma_j\setminus
(\Gamma_1\cup \dots \cup \Gamma_{j-1})$. Then since $\supp
\phi_j\cap \supp \phi_k =\emptyset$ for every $j\neq k$ and $$
\bigcup_{j=1}^m \Gamma_j\setminus (\Gamma_1\cup \dots \cup
\Gamma_{j-1}) = \bigcup_{j=1}^m \Gamma_j =\rr d \back 0 $$ it
follows that $\sum \phi_j=1$ in $\rr d\back 0$. Let $F=K*f$ and
$F_j=K*(f'_j-f_j).$ Then $$ \sum_{j=1}^m F_j=\sum_{j=1}^m
(K*(f_j'-f_j))=K*(\sum_{j=1}^m f'_j-\sum_{j=1}^m f_j) =0. $$ Let
$$ f_j=\int F(\cdot-i\xi)\phi_j(\xi)\, d\xi $$ and
\begin{equation}\label{fjk}
f_{jk} =\int F_j(\cdot-i\xi)\phi_k(\xi)\, d\xi-\int
F_k(\cdot-i\xi)\phi_j(\xi)\, d\xi.
\end{equation}
 Then it follows by straight-forward calculations that $f_j'=f_j + \sum_{k=1}^m f_{jk}$ and $f_{jk}=-f_{kj}$.
More precisely we have that
\begin{multline*}
\sum_{k=1}^m f_{jk} =\int F_j(\cdot-i\xi)\sum_{k=1}^m
\phi_k(\xi)\,d\xi-\int \sum_{k=1}^m
F_k(\cdot-i\xi)\phi_j(\xi)\,d\xi\\[1 ex]= \int F_j(\cdot-i\xi)\,
d\xi =f_j'-f_j.
\end{multline*}
From Theorem \ref{huvudsats} and Lemma \ref{andra riktningen} it
follows that \eqref{1} holds using that $\phi_j$ has support in
$\Gamma_j\cap S^{d-1}$ and letting $d\mu_j(\xi)=\phi_j(\xi)\,
d\xi$. Use the measure defined above and treat the integrals on
the right-hand side of \eqref{fjk} separately. By using the
arguments above we see that the wave-front sets of these integrals
are contained in $$(\WF_{\FB}(f_j)\cup \WF_{\FB}(f'_j))\cap(\rr
d\times \Gamma_k)$$ and $$(\WF_{\FB}(f_k)\cup
\WF_{\FB}(f'_k))\cap(\rr d\times \Gamma_j)$$  respectively. Now
\eqref{2} follows immediately from this together with the fact
that $f_j$ and  $f'_j$ satisfies \eqref{1}.
\end{proof}

\section{Wave-front sets of weighted Fourier BF-types}\label{sec2'}

In this section we consider weighted Fourier BF-spaces and prove results analogous to the non-weighted case. We start by assuming that $\mathscr B$ is a translation invariant BF-space and $\omega\in\mathscr P(\rr d)$. Then let $\mathscr B_1=\mathscr B(\omega)$. By the following lemma we see that there is no restriction to assume that $\omega$ is $v_0$-moderated for some $v_0\in\mathscr P(\rr d)$ which is submultiplicative. 

\par

\begin{lemma}
Assume that $\omega\in\mathscr P(\rr d)$. Then there exists $v_0\in \mathscr P(\rr d)$ such that $v_0$ is submultiplicative and 
$$
\omega(x)\leq C v_0(x),
$$ 
where the constant $C>0$ is independent of $x\in \rr d$.  
\end{lemma}
\begin{proof}
Assume that $\omega\in \mathscr P(\rr d)$. Then we can choose constants $N$ and $C$ large enough such that 
$$
\omega(x)\leq C w(0)\eabs x^N.
$$ 
Note that $C$ and $N$ do not depend on $x\in \rr d$. From the fact that
$$
\eabs{x+y}\leq 2 \eabs x\eabs y,
$$ 
for every $x,y\in \rr d$ it follows that $\eabs x^N$ is submultiplicative and polynomially moderated.
\end{proof}

\par

Assume that $\omega, v\in \mathscr P(\rr d)$ and that $v$ is submultiplicative. By the previous lemma we may choose $v_0\in\mathscr P(\rr d)$ such that $v_0$ is submultiplicative and $\omega$ is $v_0$-moderate. Also assume that $\mathscr B$ is a translation invariant BF-space on $\rr d$ with respect to $v$ and let $\mathscr B_1$ be the Banach space which consists of all $f\in L^{1}_{\loc}(\rr d)$ such that $\nm{f}{\mathscr B_1}\equiv \nm{f\omega}{\mathscr B}$ is finite. We recall from Remark 1.2 in \cite{CJT1} that $\mathscr B_1$ then is a translation invariant BF-space with respect to $v_0v$. 

\par

Next we state the main results in the weighted version. Since these results are obtained directly using the statement above together with the analogous results for the non-weighted case we give the following results without proofs.

\renewcommand{\rubrik}{Theorem \ref{huvudsats}$'$}

\begin{tom}

Assume that $f$ and $F$ satisfies the conditions in Theorem \ref{8.4.11}. Also assume that $\omega, v ,v_0\in \mathscr P(\rr d)$ are such that $v$ and $v_0$ are submultiplicative, $\omega$ is $v_0$-moderate, and assume that $\mathscr B$ is a translation invariant BF-space on $\rr d$ with respect to $v$.  Then we have that
$$
(\rr d\times S^{d-1})\cap \WF_{\FB(\omega)}(f)=\{(x,\xi); \, |\xi|=1,\, F \text{ is not in} \, \FB(\omega)\text{ at } x-i\xi\}.
$$
\end{tom}

\renewcommand{\rubrik}{Lemma \ref{andra riktningen}$'$}

\begin{tom}
Let $\mathscr B$ and $\omega$, $v$ and $v_0$ be defined as in the previous theorem. Also let $d \mu$ be a measure on $S^{d-1}$ and $\Gamma$ an open convex cone such that
\begin{equation*}
\scal y \xi <0\, \text{ when } \, 0\neq y\in \overline{\Gamma},\, \xi\in \supp d\mu.
\end{equation*}
If $F$ is  analytic in $\Omega$ and satisfies \eqref{rel1fF}, then
$$ F_1(z)=\int F(z+i\xi) \, d\mu(\xi) $$ is analytic and
$|F_1(z)|\leq C'(1+|\re z|)^a|\im z|^{-b}$ when $\im z\in \Gamma$
and $|\im z|$ is small enough.

\par

For every measure $d\mu$ on $S^{d-1}$ we have
\begin{equation}\label{omvandvikt}
\WF_{\FB(\omega)}(F_\mu)\subset \{(x,\zeta);\, -\zeta/|\zeta| \in \supp d\mu \text{ and } F\not\in \FB(\omega) \, \text{at}\, x-i\zeta/|\zeta|\}
\end{equation}
Here $F_\mu=\int F(\cdot + i\xi)\, d\mu(\xi)$.
\end{tom}

\par

\renewcommand{\rubrik}{Corollary \ref{8.4.13'}$'$}

\begin{tom}
Let $\mathscr B$ and $\omega$, $v$ and $v_0$ be defined as in Theorem \ref{huvudsats}$''$.
Also let $\Gamma_1,\dots ,\Gamma_m$ be closed cones in $\rr d\back 0$
such that $$ \bigcup_{j=1}^m \Gamma_j = \rr d\back 0. $$ For every
$f\in \mathscr S'(\rr d)$ there exists a  decomposition
$f=\sum_{j=1}^m f_j$, where $f_j\in \mathscr S'$ and
\begin{equation}\label{1vikt}
\WF_{\FB(\omega)} (f_j) \subseteq \WF_{\FB(\omega)} (f) \cap (\rr d \times \Gamma_j).
\end{equation}
If there exists another decomposition $f=\sum_{j=1}^m f_j'$ which
also satisfies the conditions above, then $f_j'=f_j + \sum_{k=1}^m
f_{jk}$ where $f_{jk}\in \mathscr S'$, $f_{jk}=-f_{kj}$ and
\begin{equation}\label{2vikt}
\WF_{\FB(\omega)} (f_{jk}) \subset WF_{\FB(\omega)}(f) \cap (\rr d \times (\Gamma_j \cap \Gamma_k)).
\end{equation}
\end{tom}

\par

\section{Wave-front sets of inf type}\label{sec2''}

\par

In this section we show analogous results for wave-front sets of inf types. We recall the definitions of these types of wave-front sets from Coriasco, Johansson and Toft \cite{CJT1}.
Let $\mathscr B_j$ be a translation invariant BF-space on $\rr d$ and $\omega_j\in \mathscr P(\rr d)$, when $j$ belongs to some index set $J$, and consider the array of spaces, given by 
\begin{equation}\label{B_j}
(\mathcal B_j)\equiv (\mathcal B_j)_{j\in J},\quad \text{where} \quad \mathcal B_j= \FB_j(\omega_j), \quad j\in J.
\end{equation} 
\par

 We recall that the
wave-front sets of inf types
$\WF^{\inf}_{(\mathcal B_j)}(f)=\WF^{\inf}_{(\FB_j(\omega_j))}(f)$ consists of all pairs $(x_0,\xi_0)\in \rr d
\times \rr d\back 0$ such that for every
open conical neighbourhood $\Gamma_{\xi_0}$ of $\xi_0$, every
$\fy\in C^{\infty}_0$ with $\fy=1 $ in some open neighbourhood $X$
of $x_0$ and for every $j\in J$  it holds that $$ |\fy f|_{\FB_j(\Gamma_{\xi_0})}\equiv \nm{\mathscr F (\fy f)
\chi_{\Gamma_{\xi_0}}}{\mathscr B_j(\omega_j)} =\infty. $$ Here $\chi_{\Gamma_{\xi_0}}$ is the characteristic
function of $\Gamma_{\xi_0}$.

\par

Before stating analogous results to those for wave-front sets of Fourier BF-spaces we compare the wave-front sets of Fourier BF-spaces with the wave-front sets of inf types defined above.

\par

Since $(x_0,\xi_0)\in\WF^{\inf}_{(\mathcal B_j)}(f)$ if and only if $(x_0,\xi_0)\in\WF_{\mathcal B_j}(f)$ for every $j\in J$, it follows that 
\begin{equation}
\WF^{\inf}_{(\mathcal B_j)}(f)=\bigcap_j\WF_{\mathcal B_j}(f)
\end{equation} 

\par

\renewcommand{\rubrik}{Theorem \ref{huvudsats}$''$}

\begin{tom}
Let $\mathscr B_j$ be a translation invariant BF-space on $\rr d$ and $\omega\in \mathscr P(\rr d)$ for every $j\in J$. Also let $\mathcal B_j$ be defined as in \eqref{B_j} and let $f$ and $F$ satisfy the conditions in Theorem \ref{8.4.11}. Then we have that
$$
(\rr d\times S^{d-1})\cap \WF^{\inf}_{(\mathcal B_j)}(f)=\{(x,\xi); \, |\xi|=1,\, F \text{ is not in} \, \bigcup_j\mathcal B_j \text{ at } x-i\xi\}.
$$
\end{tom}

\begin{proof}
We have that 
\begin{multline}
(\rr d\times S^{d-1})\cap \WF^{\inf}_{(\mathcal B_j)}(f)= \bigcap_j((\rr d\times S^{d-1})\cap \WF_{\mathcal B_j}(f))\\[1 ex]=\bigcap_j\{(x,\xi); \, |\xi|=1,\, F \text{ is not in} \, \mathcal B_j \text{ at } x-i\xi\}\\[1 ex]=\{(x,\xi); \, |\xi|=1,\, F \text{ is not in} \, \bigcup_j\mathcal B_j \text{ at } x-i\xi\}.
\end{multline}
The proof is complete
\end{proof}

\begin{lemma}
Let $d \mu$ be a measure on $S^{d-1}$ and $\Gamma$ an open convex cone such that
\begin{equation*}
\scal y \xi <0\, \text{ when } \, 0\neq y\in \overline{\Gamma},\, \xi\in \supp d\mu.
\end{equation*}
If $F$ is  analytic in $\Omega$ and satisfies \eqref{rel1fF}, then
$$ F_1(z)=\int F(z+i\xi) \, d\mu(\xi) $$ is analytic and
$|F_1(z)|\leq C'(1+|\re z|)^a|\im z|^{-b}$ when $\im z\in \Gamma$
and $|\im z|$ is small enough.

\par

For every measure $d\mu$ on $S^{d-1}$ we have
\begin{equation}\label{omvandinf}
\WF^{\inf}_{(\mathcal B_j)}(F_\mu)\subset \{(x,\zeta);\, -\zeta/|\zeta| \in \supp d\mu \text{ and } F\not\in  \bigcup_j \mathcal B_j \, \text{at}\, x-i\zeta/|\zeta|\}
\end{equation}
Here $F_\mu=\int F(\cdot + i\xi)\, d\mu(\xi)$.
\end{lemma}

\par

The following Corollary is an analogue to Corollary 8.4.13 in Hörmander \cite{Ho1}.

\begin{cor}\label{8.4.13'''}
Let $\Gamma_1,\dots ,\Gamma_m$ be closed cones in $\rr d\back 0$
such that $$ \bigcup_{j=1}^m \Gamma_j = \rr d\back 0. $$ For every
$f\in \mathscr S'(\rr d)$ there exists a  decomposition
$f=\sum_{j=1}^m f_j$, where $f_j\in \mathscr S'$ and
\begin{equation}\label{1inf}
\WF^{\inf}_{(\mathcal B_j)} (f_j) \subseteq \WF^{\inf}_{(\mathcal B_j)} (f) \cap (\rr d \times \Gamma_j).
\end{equation}
If there exists another decomposition $f=\sum_{j=1}^m f_j'$ which
also satisfies the conditions above, then $f_j'=f_j + \sum_{k=1}^m
f_{jk}$ where $f_{jk}\in \mathscr S'$, $f_{jk}=-f_{kj}$ and
\begin{equation}\label{2inf}
\WF^{\inf}_{(\mathcal B_j)} (f_{jk}) \subset \WF^{\inf}_{(\mathcal B_j)}(f) \cap (\rr d \times (\Gamma_j \cap \Gamma_k)).
\end{equation}
\end{cor}

\par

\section{Wave-front sets of modulation space types}\label{sec2'''}
In this section we show that the results obtained for wave-front sets of Fourier BF-space types also hold for wave-front sets of modulation space types.

\par

We start by defining general types of modulation spaces. Let (the window) $\phi\in \mathscr S'(\rr d)\back 0$ be fixed, and let $f\in\mathscr S'(\rr d)$. Then the short-time Fourier transform $V_\phi f$ is the element in $\mathscr S'(\rr {2d})$, defined by the formula 
$$
(V_\phi f)(x,\xi)\equiv \mathscr F(f\cdot \overline{\phi(\cdot -x)})(\xi).
$$
We usually assume that $\phi\in \mathscr S(\rr d)$, and in this case the short-time Fourier transform $(V_\phi f)$ takes the form 
$$
(V_\phi f )(x,\xi)=(2 \pi)^{-d/2} \int_{\rr d} f(y)\overline{\phi(y-x)} e^{-\scal y \xi}\, dy,
$$
when $f\in \mathscr S(\rr d)$.

Now let $\mathscr B$ be a translation invariant BF-space on $\rr {2d}$, with respect to $v\in \mathscr P(\rr {2d})$. Also let $\phi\in \mathscr S(\rr d)\back 0$ and $\omega \in \mathscr P(\rr {2d})$ be such that $\omega$ is $v$-moderate. Then the modulation space $M(\omega)=M(\omega,\mathscr B)$ is a Banach space with the norm 
\begin{equation}
\nm{f}{M(\omega,\mathscr B)}\equiv\nm{V_\phi f\omega}{\mathscr B}
\end{equation} 
(cf. \cite{FeGro1}).

\par

 Assume that
$\omega\in \mathscr P(\rr {2d})$.  We recall that the 
wave-front sets of modulation space types $\WF_{M(\omega,\mathscr B)}(f)$ consists of all pairs $(x_0,\xi_0)\in \rr d
\times \rr d\back 0$ such that $$ |\fy f|_{M(\omega,\mathscr B,
\Gamma_{\xi_0})}\equiv \nm{ V_\phi(\fy f)
\chi_{\Gamma_{\xi_0}}\omega}{\mathscr B} =\infty, $$ for every
open conical neighbourhood $\Gamma_{\xi_0}$ of $\xi_0$, and
$\fy\in C^{\infty}_0$ with $\fy=1 $ in some open neighbourhood $X$
of $x_0$. Here $\chi_{\Gamma_{\xi_0}}$ is the characteristic
function of $\Gamma_{\xi_0}$. 
It can also be showed that wave-front sets of modulation space types and wave-front sets of Fourier BF-types coincide. More precisely, let 
\begin{equation}\label{B0}
\mathscr B_0= \{f\in \mathscr S'(\rr d) :\, \fy\otimes f\in \mathscr B\}.
\end{equation}
Then $\mathscr B_0$ is a translation invariant BF-space on $\rr d$, which is independent of the choice of $\fy$. Furthermore $M(\omega,\mathscr B)$ and $\FB_0$ are locally the same and 
$$
\WF_{\FB_0(\omega)}(f)=\WF_{M(\omega,\mathscr B)}(f).
$$
(Cf. Coriasco, Johansson and Toft
\cite{CJT1}.) 

By using the previous results in combination with this we obtain the following results. 

\begin{defn}\label{MB i x}
Assume that $f\in \mathscr D'(\rr d)$, $\mathscr B$ is a
translation invariant BF-space and $\omega \in \mathscr P(\rr {2d})$.
Then $f\in M(\omega,\mathscr B)$ at $x_0$ if and only if there is
some neighbourhood $X$ of $x_0$ such that for some $\fy\in
C^{\infty}_0$ with $\fy\equiv 1$ in $X$ we have that $\fy f \in
M(\omega,\mathscr B)$.
\end{defn}
We recognize by the arguments before that since the definition above only concerns local properties it holds that  $f\in M(\omega,\mathscr B)$ at $x_0$ if and only if $f\in \FB_0(\omega)$ at $x_0$, where $\mathscr B_0$ is given by \eqref{B0}.

We note that if $f$ belongs to $M(\omega,\mathscr B)$ at $x_0$ then $(x_0,\xi_0)\notin\WF_{M(\omega,\mathscr B)}(f)$ for any $\xi_0\in \rr d\back 0.$

\begin{defn}
For $f\in \mathscr D'(X)$ the singular support $\singsupp_{M(\omega,\mathscr B)} f$ is the smallest closed subset of $X$ such that $f$ is in $M(\omega,\mathscr B)$ in the complement.
\end{defn}

\begin{thm}
Assume that $f\in \mathscr D'(\rr d)$, $\mathscr B$ is a
translation invariant BF-space and $\omega\in\mathscr P (\rr
{2d})$. The projection of $\WF_{M(\omega,\mathscr B)} (f)$ in $X$ is equal to
$\singsupp_{M(\omega,\mathscr B)} f$.
\end{thm}

\begin{proof}
(a) Assume that $x_0\not\in \singsupp_{M(\omega,\mathscr B)}(f)$. Then $f$ belongs to $M(\omega,\mathscr B)$ at $x_0$. This implies that $(x_0,\xi_0)\not\in\WF_{M(\omega,\mathscr B)} (f)$, for any $\xi_0\in \rr d\back 0$.

\medspace

(b) Assume that $(x_0,\xi_0)\not\in\WF_{M(\omega,\mathscr B)} (f)$ for all
$\xi_0\in \rr d\back 0$. Then we can choose a neighbourhood $K$ of
$x_0$ such that $\WF_{M(\omega,\mathscr B)} (f)\cap (K\times \rr
d)=\emptyset$. This implies that we can choose a function
$\fy_{x_0}\in C^\infty _0$ which is equal to $1$ in a
neighbourhood $X$ of $x _0$ such that $\fy_{x_0} f \in M(\omega,\mathscr B)$. Hence $x_0\not\in \singsupp_{M(\omega,\mathscr B)}
(f)$.
\end{proof}

\par

Next theorem is analogous to Theorem \ref{huvudsats}.

\renewcommand{\rubrik}{Theorem \ref{huvudsats}$'$}

\begin{tom}

Assume that $f$ and $F$ satisfy the conditions in Theorem \ref{8.4.11}. Also let $\mathscr B$ be a translation invariant BF-space and $\omega \in \mathscr P(\rr {2d})$. Then we have that
$$
(\rr d\times S^{d-1})\cap \WF_{M(\omega,\mathscr B)}(f)=\{(x,\xi); \, |\xi|=1,\, F \text{ is not in} \, M(\omega,\mathscr B)\text{ at } x-i\xi\}.
$$
\end{tom}

We remark that $F$ is in $M(\omega,\mathscr B)$ at $x-i\xi$ if for some neighbourhood $V$ of $(x,\xi)$ there exists some localization $\fy \in C^{\infty}_0$ with $\fy=1$ in $V$ such that $\fy f\in M(\omega,\mathscr B)$. 
\par

\begin{proof}
Let $\mathscr B_0$ be defined as before. Then it follows that 
$$
(\rr d\times S^{d-1})\cap \WF_{M(\omega,\mathscr B)}(f)=(\rr d\times S^{d-1})\cap \WF_{\FB_0(\omega)}(f).
$$
From the result in the previous section we also have that 
$$
(\rr d\times S^{d-1})\cap \WF_{\FB_0(\omega)}(f)=\{(x,\xi); \, |\xi|=1,\, F \text{ is not in} \, \FB_0(\omega) \text{ at } x-i\xi\}.
$$
Now since the right-hand side only concern local properties and $\FB_0(\omega)$ and $M(\omega,\mathscr B)$ are locally the same it follows that 
\begin{multline*}
\{(x,\xi); \, |\xi|=1,\, F \text{ is not in} \, \FB_0(\omega) \text{ at } x-i\xi\}\\[1 ex]=\{(x,\xi); \, |\xi|=1,\, F \text{ is not in} \, M(\omega,\mathscr B) \text{ at } x-i\xi\}.
\end{multline*}
This completes the proof.
\end{proof}

By arguments given before it is obvious that Lemma \ref{andra riktningen} and Corollary \ref{8.4.13'} hold also for modulation spaces instead of Fourier BF-spaces. We therefore state the following results without proofs.

\renewcommand{\rubrik}{Lemma \ref{andra riktningen}$'$}

\begin{tom}
Let $\mathscr B$ be a translation invariant BF-space and $\omega\in \mathscr P(\rr {2d})$. Also let $d \mu$ be a measure on $S^{d-1}$ and $\Gamma$ an open convex cone such that
\begin{equation*}
\scal y \xi <0\, \text{ when } \, 0\neq y\in \overline{\Gamma},\, \xi\in \supp d\mu.
\end{equation*}
If $F$ is  analytic in $\Omega$ and satisfies \eqref{rel1fF}, then
$$ F_1(z)=\int F(z+i\xi) \, d\mu(\xi) $$ is analytic and
$|F_1(z)|\leq C'(1+|\re z|)^a|\im z|^{-b}$ when $\im z\in \Gamma$
and $|\im z|$ is small enough.

\par

For every measure $d\mu$ on $S^{d-1}$ we have
\begin{equation}\label{omvandmod}
\WF_{M(\omega,\mathscr B)}(F_\mu)\subset \{(x,\zeta);\, -\zeta/|\zeta| \in \supp d\mu \text{ and } F\not\in M(\omega,\mathscr B) \, \text{at}\, x-i\zeta/|\zeta|\}
\end{equation}
Here $F_\mu=\int F(\cdot + i\xi)\, d\mu(\xi)$.
\end{tom}

\par

\renewcommand{\rubrik}{Corollary \ref{8.4.13'}$'$}

\begin{tom} Let $\mathscr B$ be a translation invariant BF-space and $\omega\in \mathscr P(\rr {2d})$. Also let $\Gamma_1,\dots ,\Gamma_m$ be closed cones in $\rr d\back 0$
such that $$ \bigcup_{j=1}^m \Gamma_j = \rr d\back 0. $$ For every
$f\in \mathscr S'(\rr d)$ there exists a  decomposition
$f=\sum_{j=1}^m f_j$, where $f_j\in \mathscr S'$ and
\begin{equation}\label{1M}
\WF_{M(\omega,\mathscr B)} (f_j) \subseteq \WF_{M(\omega,\mathscr B)} (f) \cap (\rr d \times \Gamma_j).
\end{equation}
If there exists another decomposition $f=\sum_{j=1}^m f_j'$ which
also satisfies the conditions above, then $f_j'=f_j + \sum_{k=1}^m
f_{jk}$ where $f_{jk}\in \mathscr S'$, $f_{jk}=-f_{kj}$ and
\begin{equation}\label{2M}
\WF_{M(\omega,\mathscr B)} (f_{jk}) \subset WF_{M(\omega,\mathscr B)}(f) \cap (\rr d \times (\Gamma_j \cap \Gamma_k)).
\end{equation}
\end{tom}

\par

\section{Some additional properties}\label{sec3}
In this section we prove some further properties for the wave-front sets of Fourier Banach types using results from the previous section.

\par

\begin{thm}\label{8.4.15'}
Let $f\in \mathscr D'(X)$, $X\subseteq \rr d$, and $\WF_{\FB}(f)\subseteq X\times \Gamma^\circ$, where $\Gamma^\circ$ is the dual of an open convex cone $\Gamma$. If $\overline{X_1}\subseteq X$ and $\Gamma_1$ is an open convex with $\overline{\Gamma_1}\subseteq \Gamma\cup \{0\}$, then there exists a function $F$ that is analytic in $\{x+iy;\, x\in X_1, \, y\in \Gamma_1,\, |y|<\gamma\}$, such that
$$
|F(x+iy)|<C|y|^{-N}, \qquad y\in \Gamma_1,\qquad x\in X_1,
$$
and such that the limit of $F(\cdot-iy)$ in $\Gamma_1$, when $y\to 0$, differs from $f$ by an element in $\FB(X_1)$.
\end{thm}

\begin{proof}
Set $v=\chi f$ where $\chi \in C^\infty_0$ is equal to $1$ in $X_1$. If $V=K*v$ is defined as in Theorem \ref{huvudsats}, then
$$
\WF_{\FB}(v)=\WF_{\FB}(\chi f)\subseteq \WF_{\FB}(f) \subseteq X\times\Gamma^\circ
$$
gives
$$
\complement (X\times \Gamma^\circ) \subseteq \complement \WF_{\FB} (v).
$$
From this follows that
$$
X_1\times \complement \Gamma^\circ \subseteq \complement \WF_{\FB}(v).
$$
Then Theorem \ref{huvudsats} implies that $V\in \FB$ at every point in $X_1+ i(S^{d-1}\cap \complement (-\Gamma^\circ))$. Choose an open set $M$ with $\Gamma^\circ\cap S^{d-1}\subseteq M\subseteq S^{d-1}$ and where $\overline M$ belongs to the interior of $\Gamma^\circ$. Then $v=v_1+v_2$ where
$$
v_1=\int_{-\xi\not\in M}V(\cdot+i\xi)\, d\xi
$$
belongs to $\FB$ in $X_1$ and $v_2$ is the boundary value of the analytic function
$$
F(z)=\int_{-\xi\in M}V(z+i\xi)\, d\xi,\qquad \im z\in \Gamma_1,\qquad |\im z|<\gamma.
$$
Lemma \ref{andra riktningen} completes the proof.
\end{proof}

As mentioned before we have that $\WF_{\FB}(f) \subseteq \WF_A(f)$.
In the following proposition we describe a relation between the wave-front sets of Fourier Banach function types and analytic wave-front sets.

\begin{prop}\label{FB och analytisk}
Let $\mathscr B$ be a translation invariant BF-space and $ f\in\mathscr D'(\rr d)$. Then
\begin{equation*}
\WF_{\FB} (f) =\bigcap_{g\in \FB} \WF_A(f-g).
\end{equation*}
\end{prop}
For the proof we need the following Lemmas which are extensions of Proposition 1.5 and Lemma 1.6 in \cite{PTT2}.

\begin{lemma}\label{produkt}
Let $X\subseteq \rr d$ be open and $\mathscr B$ be a translation invariant BF-space. Then the map $(f_1, f_2)\mapsto f_1f_2$ from $\mathscr S (\rr d)\times \mathscr S (\rr d)$ to $\mathscr S (\rr d)$ extends uniquely to continuous mapping from $\FB (\rr d)\times \mathscr F L^1_{(v)}(\rr d)$ to $\FB(\rr d)$.
 \end{lemma}

\begin{proof}
$(1)$ Let $f_1\in \FB(\rr d)$ and $f_2 \in \mathscr S(\rr d)$. By Minkowski's inequality it follows that
\begin{equation*}
\nm {f_1f_2}{\FB} =\nm{\widehat {f_1}*\widehat {f_2}}{\mathscr B}\leq C\nm {\widehat {f_1}}{\mathscr B} \nm {\widehat {f_2}}{L^1_{(v)}}.
\end{equation*}
The assertion $(1)$ now follows from this estimate and the fact that $\mathscr S(\rr d)$ is dense in $\mathscr F L^1_{(v)}$.
\end{proof}

\begin{lemma}\label{utvavPTT2}
Let $X\subseteq \rr d$ be open, $f\in \mathscr D'(X)$ and let $\mathscr B$ be a translation invariant BF-space. Also let $(x_0,\xi_0)\in X\times \rr d\back 0$. Then the following conditions are equivalent:
\begin{enumerate}
\item $(x_0,\xi_0)\not\in \WF_{\FB} (f)$;

\item there exists $g\in \FB(\rr d)$ ($g\in \FB_{\loc}(X)$) such that  $(x_0,\xi_0)\not\in \WF(f-g)$;

\item there exists $g\in \FB(\rr d)$ ($g\in \FB_{\loc}(X)$) such that  $(x_0,\xi_0)\not\in \WF_A(f-g)$.
\end{enumerate}
\end{lemma}

\begin{proof}
In this proof we use the same ideas as in \cite[Proposition 8.2.6]{Hrm-nonlin} (see also \cite{PTT2} and \cite{PTT3}). We may assume that $g\in \FB_{\loc}(X)$ in $(2)$ and $(3)$ since the wave-front sets concern local properties. Assume that $(2)$ holds. We can then find an open subset $X_0$ of $X$ and some open cone $\Gamma=\Gamma_{\xi_0}$ and a sequence $\fy_N\in C^\infty_0$ such the $\fy_N(f-g)=f-g$ on $X_0$ and
\begin{equation}\label{uppskattning1}
|\mathscr F (\fy_N(f-g))(\xi)| \leq C_{N,\fy_N}\eabs{\xi}^{-N},\qquad N=1,2,\dots,\, \xi\in \Gamma.
\end{equation}
In particular it follows that if $N_0$ is chosen large enough, then $|\fy_N (f-g)|_{\FB(\Gamma)}$ is finite for every $N>N_0$. Since $g\in \FB(\rr d)$, it follows by Lemma \ref{produkt} that $|\fy_N g|_{\FB(\Gamma)}$ is finite for every $\fy_N$. Then $|\fy_N f|_{\FB(\Gamma)}$ is finite  for every $N>N_0$ and $(1)$ holds.

\par

Conversely, if $(x_0,\xi_0)\not\in \WF_{\FB}(f)$, then there exist an open neighbourhood $X_0$ of $x_0$ and an open conical neighbourhood $\Gamma$ of $\xi_0$  such that
$$|\fy f|_{\FB(\Gamma)}< \infty,$$
when $\fy \in C^\infty_0$, in view of Theorem 3.2 in \cite{CJT1}.

\par

Let $\fy_1,\fy\in C^\infty_0(X_0)$ be chosen such that $\fy(x_0)\neq 0$ and $\fy_1=1$ in the support of $\fy$. Furthermore let $\widehat g =\mathscr F (\fy_1f)$ in $\Gamma$ and otherwise $0$. Then $g\in \FB(\rr d)$.

\par

By \cite[Lemma 8.1.1]{Ho1} and its proof, it follows that
$$
|\mathscr F (\fy_1(\fy f-g))(\xi)|<C_N\eabs \xi^{-N},\qquad N=0,1,2\cdots,
$$
when $\xi\in \Gamma$ and $\Gamma$ is chosen sufficiently small. Since $\fy\fy_1=\fy$ we have that \eqref{uppskattning1} holds. This implies that $(x_0,\xi_0)\not\in \WF(f-g)$. This proves that $(1)$ and $(2)$ are equivalent.

\par

Since $\WF(f) \subseteq \WF_A(f)$ for each distribution $f$, it follows that $(2)$ holds if $(3)$ is fulfilled. Assume that $(2)$ holds. Then in view of of the remark before Corollary 8.4.16 in \cite{Ho1} there exists some $h\in C^\infty (X)$ such that $(x_0,\xi_0)\not\in \WF_A(f-g-h)$. Since $C^\infty \subseteq \FB_{\loc}(X)$ it follows that $g_1=g + h \in \FB_{\loc}(X)$. Hence $(3)$ holds, and the result follows.
\end{proof}

\begin{proof}[Proof of Proposition \ref{FB och analytisk}]
We start by showing that
\begin{equation*}
\WF_{\FB}(f)\subseteq \WF_A(f-g),
\end{equation*}
for every $g\in \FB$. Since $\WF_{\FB}(f-g)\subseteq \WF_A(f-g)$ it is sufficient to show that
\begin{equation}\label{steg1}
\WF_{\FB}(f)\subseteq \WF_{\FB} (f-g),
\end{equation}
for every $g\in \FB$.

\par

Assume that $(x_0,\xi_0)\not\in \WF_{\FB}(f-g)$. Then there exist $\fy_{x_0}\in C^\infty_0$ with $\fy_{x_0}(x_0)\neq 0$ and an open conical neighbourhood $\Gamma=\Gamma_{\xi_0}$ of $\xi_0$ such that
\begin{equation*}
\nm{\mathscr F(\fy_{x_0} (f-g))\chi_{\Gamma_{\xi_0}}}{\mathscr B} <\infty.
\end{equation*}
It follows by Lemma \ref{produkt} that
$$
\nm{\mathscr F(\fy_{x_0} g)\chi_{\Gamma_{\xi_0}}}{\mathscr B} <\infty
$$
for every $g\in \FB$ and then
$$
\nm{\mathscr F(\fy_{x_0} f)\chi_{\Gamma_{\xi_0}}}{\mathscr B}=\nm{\mathscr F(\fy_{x_0} (f-g))\chi_{\Gamma_{\xi_0}}}{\mathscr B} +\nm{\mathscr F(\fy_{x_0} g)\chi_{\Gamma_{\xi_0}}}{\mathscr B} <\infty.
$$
This shows that \eqref{steg1} holds. In fact, by similar calculations we can show the opposite inclusion and thereby obtain equality in \eqref{steg1}.
We have now shown that
$$
\WF_{\FB}(f) \subseteq \bigcap_{g\in \FB} \WF_A (f-g).
$$
We obtain the opposite inclusion by using Proposition \ref{utvavPTT2}. This completes the proof.
\end{proof}

\begin{cor}
If $f\in \mathscr D'(X)$ where $X$ is an interval on $\rr{}$ and if $x_0\in X$ is a boundary point of $\supp f$, then $(x_0,\pm 1)\in \WF_{\FB}(f)$.
\end{cor}

\begin{proof}
Assume for example that $(x_0,-1)\not\in \WF_{\FB}(f).$ Then we can find $F$ analytic in $\Omega=\{z;\, \im z>0,\, |z-x_0|<r\}$  with boundary value $f$. There is an interval $I\subseteq (x_0-r,x_0+r)$ where $f=0$. By Theorem 3.1.12 and 4.4.1 in Hörmander \cite{Ho1} $F$ can be extended analytically across $I$ so that $F=0$ below $I$. Thus the uniqueness of analytic continuation gives  $F=0$, hence $f=0$ in $(x_0-r,x_0+r)$. This contradicts that $x_0$ is a boundary point of $\supp f$ and proves the corollary.
\end{proof}

The result in the following lemma follows directly from Lemma 8.4.17 in Hörmander.
\begin{lemma} If $f\in \mathscr S'$ then $\WF_{\FB}(f)\subseteq \rr d \times F$ where $F$ is the limit cone of $\supp \widehat f $ at infinity, consisting of all limits of sequences $t_jx_j$ with $x_j\in \supp \widehat f$ and $0<t_j \to 0$.
\end{lemma}

\par

Next we give some computational rules for wave-front sets of Fourier Banach function types. We prove that some of  the rules that Hörmander obtained for classical and analytical wave-front sets in \cite{Ho1} holds also for wave-front sets of Fourier Banach function types. For completeness we give the proofs which are similar to those of the analogous results in \cite{Ho1}.

\begin{thm}
Let $X\subseteq \rr{d_1}$ and $Y\subseteq \rr{d_2}$ be open. Also let $f:X\to Y$ be a real analytic map with normal set $N_f$. Then
\begin{equation}\label{linjavb}
\WF_{\FB} (f^*g)\subseteq f^*\WF_{\FB}(g),\quad \text{if} \quad g\in\mathscr D'(Y),\quad N_f\cap\WF_{\FB}(g)=\emptyset
\end{equation}
\end{thm}

\begin{proof}
Assume that there exists an analytic function $\Phi$ in
$$
\Omega=\{y' +iy'';\, y'\in Y,\, y''\in \Gamma,\, |y''|<\gamma\},
$$
where $\Gamma$ is an open convex cone, such that
$$
|\Phi(y'+iy'')|\leq C|y''|^{-N} \qquad \text{in}\; \Omega,\qquad g=\lim_{\Gamma\ni y\to 0} \Phi(\cdot+iy).
$$
Let $\Gamma^\circ$ be the dual $\Gamma$. Then by the arguments in the proof of Theorem 8.5.1 in Hörmander \cite{Ho1} it follows that $\WF_A(g)\subseteq Y \times \Gamma^\circ$. If we assume that $x_0\in X$ and that $\ltrans f'(x_0)\eta\neq 0$, $\eta \in \Gamma^\circ\back 0$, then $\ltrans f'(x_0)\Gamma^\circ$ is a closed, convex cone and
$$
\WF_A(f^*g)|_{x_0}\subseteq \{(x_0,\ltrans f'(x_0)\eta);\, \eta\in \Gamma^\circ\back 0\}.
$$
Next by using Corollary \ref{8.4.13'} and Theorem \ref{8.4.15'} it follows that any distribution $g$ can be written as a finite sum $\sum g_j$ where each term either belongs to $\FB$ in a neighbourhood of $f(x_0)$ of satisfies the hypotheses above with some $\Gamma_j$ such that $\Gamma^\circ_j$ is small and intersects $\WF_{\FB}(g)|_{f(x_0)}$. By the hypotheses $\ltrans f'(x_0)\eta\neq 0$ when $(f(x_0),\eta) \in \WF_{\FB}(f)$. We then conclude that
$$
\WF_{\FB}(f^*g)|_{x_0}\subseteq \{(x_0, \ltrans f'(x_0)\eta),\, \eta\in \bigcup \Gamma_j^\circ\}.
$$
This implies \eqref{linjavb}.
\end{proof}

\begin{thm}\label{split}
Let $f\in \mathscr E'(\rr d)$. Split the coordinates in $\rr d$ into two groups $x'=(x_1,\dots,x_{d_1})$ and $x''=(x_{d_1+1},\dots,x_d)$, and set
$$
f_1(x')=\int f(x',x'')\, dx''
$$
Then
$$
\WF_{\FB} (f_1)\subseteq \{(x',\xi');\, (x',x'',\xi',0)\in \WF_{\FB}(f) \qquad \text{for some} \, x''\}.
$$
\end{thm}

\begin{proof}
We use the same notation as in Theorem \ref{huvudsats}. Then
\begin{equation}
\scal f {\phi\otimes \psi} = \int_{|\omega|=1}\scal{U(\cdot+i\omega)}{\phi\otimes \psi}\,d\omega,
\end{equation}
for $\phi\in C^\infty_0(\rr {d_1})$ and $\psi\in C^\infty_0 (\rr {d-d_1})$. Take $\psi(x'')=\chi(\delta x'')$ where $\chi=1$ in the unit ball and let $\delta\to 0$. $U$ is decreasing exponentially at infinity and therefore it follows that
$$
\scal {f_1} \phi =\int_{|\omega|=1}\scal{U(\cdot+i\omega)}{\phi\otimes 1}\,d\omega= \int_{|\omega|=1}\scal{U_1(\cdot+i\omega')}{\phi}\,d\omega
$$
where
$$
U_1(z')=\int U(z',x'')\, dx''=\int U(z',x''+iy'')\, dx'', \qquad |\im z'|^2+|y''|^2<1
$$
is an analytic when $|\im z'|<1$, which is bounded by $C(1-|\im z'|)^{-N}$. If $|\omega_0'|<1$ and $(x',x'',\omega_0')\not\in \WF_{\FB}(f)$ for every $x''\in \rr {d-d_1}$, then $u_1\in\FB $ at $x'-i\omega'_0$. Hence Lemma \ref{andra riktningen} implies that $(x',\omega_0')\not\in \WF_{\FB}(f).$
\end{proof}

\begin{thm}
Let $X\subseteq \rr {d_1}$ and $Y\subseteq \rr {d_2}$ be open sets and $K\in \mathscr D'(X\times Y)$ be a distribution such that the projection $\supp K\to X$ is proper. If $f\in \FB(Y)$ then
$$
\WF_{\FB} (\mathscr K f)\subseteq \{(x,\xi); \, (x,y,\xi,0)\in \WF_{\FB}(K) \, \text{for some}\, y\in \supp f\}.
$$
Here $\mathscr K$ is the linear operator with kernel $K$.
\end{thm}

\begin{proof}
Replace $K$ by $K(1\otimes f)$ and assume that $f=1$. Without changing $K$ over a given compact subset of $X$ we may replace $K$ by a distribution of compact support, and then the statement is identical to Theorem \ref{split}.
\end{proof}

\end{document}